\documentclass{styles/svproc}
\usepackage{amsmath}
\usepackage{graphicx}
\usepackage{multicol}
\usepackage{footmisc}[bottom]
\usepackage{enumitem}
\usepackage{url}
\usepackage{hyperref}

\usepackage{xcolor}

\newcommand{\abs}[1]{\lvert #1 \rvert}

\newcommand{\R}{\bbbr}

\begin{document}

    \mainmatter
    
    \title{On a Free-Endpoint Isoperimetric \\
    Problem in \texorpdfstring{$\R^2$}{R^2}}
    \toctitle{On a Free-Endpoint Isoperimetric Problem in \texorpdfstring{$\R^2$}{R^2}}
    \titlerunning{On a Free-Endpoint Isoperimetric Problem in \texorpdfstring{$\R^2$}{R^2}}
    
    \author{Stanley Alama \inst{1} \and Lia Bronsard \inst{2} \and Silas Vriend \inst{3}}
    \tocauthor{Stanley Alama, Lia Bronsard, and Silas Vriend}
    \authorrunning{Alama, Bronsard \& Vriend}
    
    \institute{Department of Mathematics \& Statistics, \\
        McMaster University, Hamilton ON L8S 4L8, Canada \\
        \email{alama@mcmaster.ca} \\
        \and \email{bronsard@mcmaster.ca} \\
        \and \email{vriendsp@mcmaster.ca}
    }

    \maketitle
    
    \begin{abstract}
        Inspired by a planar partitioning problem involving multiple improper chambers, this article investigates using classical techniques what can be said of the existence, uniqueness, and regularity of minimizers in a certain free-endpoint isoperimetric problem. By restricting to curves which are expressible as graphs of functions, a full existence-uniqueness-regularity result is proved using a convexity technique inspired by work of Talenti. The problem studied here can be interpreted physically as the identification of the equilibrium shape of a sessile liquid drop in half-space (in the absence of gravity). This is a well-studied variational problem whose full resolution requires the use of geometric measure theory, in particular the theory of sets of finite perimeter, but here we present a more direct, classical geometrical approach. Conjectures on improper planar partitioning problems are presented throughout.

        \keywords{Calculus of variations $\cdot$ Isoperimetric problem $\cdot$ Geometric measure theory $\cdot$ Sets of finite perimeter $\cdot$ Sessile drop $\cdot$ Equilibrium shape $\cdot$ Improper partitioning problem} $\cdot$ Minimizing improper cluster
    \end{abstract}
    \section{Introduction}\label{sec:intro}

    \subsection{Partitioning Problems}

    The isoperimetric problem has its roots in antiquity \cite{Blasjo}. The traditional setup is as follows: given a compact domain $S$ in the plane $\R^2$ with boundary $\partial S$ of fixed length $\ell$, what is the configuration of the boundary $\partial S$ for which the area enclosed is maximal? The answer, perhaps unsurprisingly, is that $\partial S$ should be a circle. This question, together with the brachistochrone problem of Bernoulli, gave impetus to the mathematical field known as the calculus of variations.

    One can ``dualize" the traditional isoperimetric problem as follows: we may ask, given a domain $S$ of fixed area $A$ in the plane $\R^2$, what is the configuration of the boundary $\partial S$ for which the perimeter is minimal? The answer, once again, is that $\partial S$ should be a circle. In general, the term \textit{isoperimetric problem} is given to any variational problem wherein one geometrical quantity is maximized or minimized, while another is held fixed.
    
    Similarly to the dual of the isoperimetric problem given above, the following question can be posed about two domains $S_1$, $S_2$ in the plane $\R^2$, of fixed areas $A_1$, $A_2$ respectively: which configuration of the boundaries of these domains yields minimal perimeter, if such a configuration even exists? Note that we allow the two domains to reduce their total perimeter by sharing a portion of their boundaries, so the minimal-perimeter configuration need not be two disjoint discs. We remark that in this setup, the plane $\R^2$ is partitioned almost disjointly (i.e. with pairwise disjoint interiors) into three \textit{chambers}: two of finite measure, and one of infinite measure. Such problems are known as \textit{partitioning problems}, and their solutions are called \textit{minimizing clusters.}
	
	The famed double bubble conjecture asserted that given two domains in $\R^n$, each with fixed $n$-dimensional volume, the so-called standard double bubble is the configuration which uniquely minimizes perimeter. The first partial resolution of this conjecture came in work of Foisy et al.\ \cite{Foisy_etal} in 1993: they showed via ad hoc geometric methods that the standard double bubble in $\R^2$ uniquely minimizes perimeter. A resolution of the $3$-dimensional case \cite{HMRRproofDB} and then the $n$-dimensional case  \cite{Lawlor2014,Reichardt2007} followed not so long after. Beyond mere existence, further results have demonstrated the regularity and stability of such minimizing clusters \cite{Morgan1994soap,MorganWichi2002}. 
 
    The $2$-bubble problem can be generalized to the $q$-bubble problem, $q \geq 2$, but only the $3$-bubble case in $\R^2$ has been resolved to date \cite{wichiramala2004proof}. Sullivan \cite{sullivan1999geometry} conjectured that the optimal configuration in all dimensions should be a certain ``standard" bubble cluster. In 2022, Milman-Neeman \cite{milman2022structure,milman2023plateau} announced a proof of the $q$-bubble conjecture in $\R^n$ and $\bbbs^n$ for all $q \leq \mathrm{min}(5, n+1)$.  

    Another natural partitioning problem is the following: given $N$ domains $S_1, \dots, S_N$ in $\R^n$, each with  finite perimeter and fixed $n$-dimensional volumes $V_1, \dots, V_N$, and $M$ domains $U_1, \dots, U_M$ of locally finite perimeter and infinite measure, with all domains having pairwise disjoint interiors  and with the union having full measure in $\R^n$, what is the configuration of the interfaces which locally minimizes $(n-1)$-dimensional surface area? We say that a candidate configuration is \textit{locally perimeter minimizing} if it minimizes $(n-1)$-dimensional surface area relative to any other candidate configuration when tested within arbitrary compact sets, outside of which the configurations coincide. 
    
    The sets $S_1, \dots, S_N$ of finite measure are referred to as \textit{(proper) chambers}, while the sets $U_1, \dots, U_M$ of infinite measure are referred to as \textit{improper chambers}. The tuple $(S_1, \dots, S_N, U_1, \dots, U_M)$ is called an \textit{(improper) cluster}. We call such a partitioning problem \textit{improper}, and a solution to such a problem is known as a \textit{minimizing improper cluster} \cite{maggi2012sets}.

    The simplest problem with at least one chamber and multiple improper chambers is in dimension $n = 2$, with $N = 1$ chamber with a fixed area $A$, and $M = 2$ improper chambers. We conjecture that the optimal configuration for this problem is given by a \textit{vesica piscis} (the intersection of two discs with equal radii) of the desired area meeting a  pair of collinear rays at triple junctions with all angles equal to $120^\circ$; see Figure \ref{fig:optimal-solution}.

    \begin{figure}
        \centering
        
        \includegraphics[scale=0.275]{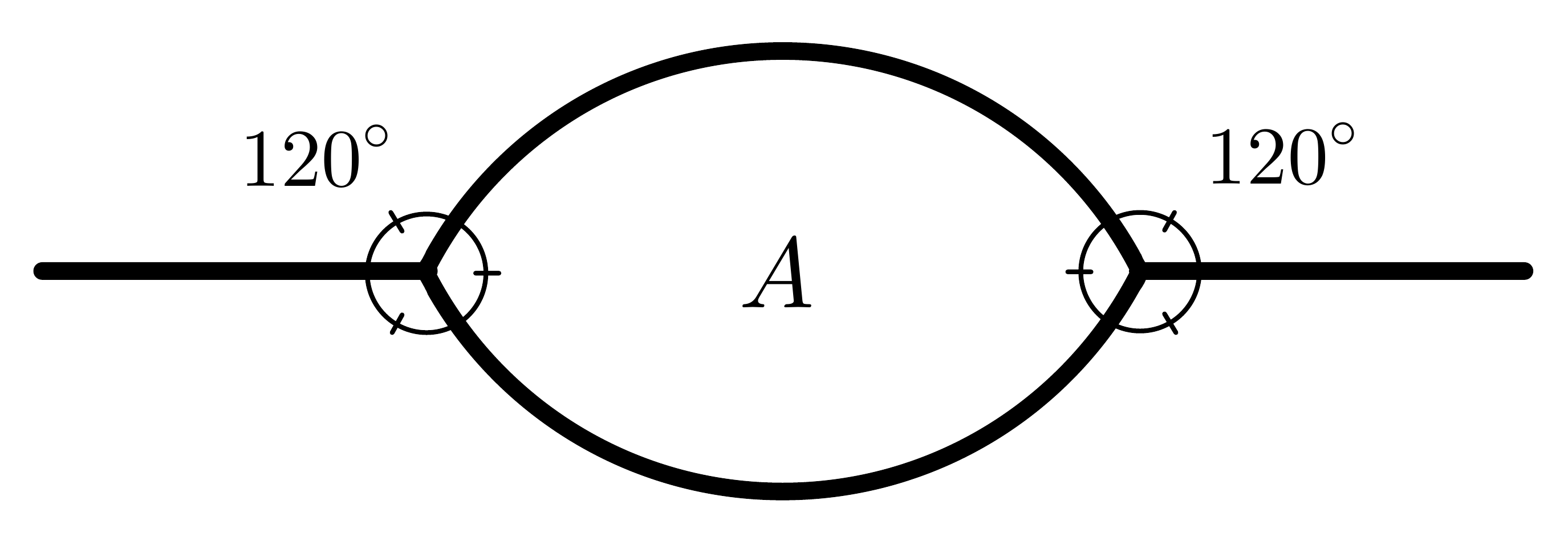}
        \caption{Conjectured optimal configuration for the improper planar partitioning problem with one chamber and two improper chambers}
        
        \label{fig:optimal-solution}
    \end{figure}
    
    This conjecture is motivated by several factors: in other partitioning
    problems, one finds that the interfaces between the chambers are minimal surfaces (i.e.\ surfaces with zero mean curvature) which meet at certain standard  junctions; see Maggi \cite{maggi2012sets}. Furthermore, in the work of Bellettini-Novaga \cite{bellettini2011curvature} and Schn\"urer et al.\ \cite{schnurer2011evolution}, this configuration appears as the limiting configuration of a planar network with two triple junctions under curve-shortening flow. One may also observe that it is conformally equivalent to the standard double-bubble via an inversion of the punctured plane $\R^2 \setminus \{(0,0)\}$ through an appropriately chosen circle. 
    
    We are therefore interested in solving the following variational problem in $\R^2$. Given a positive constant $A > 0$ (the area to enclose), we wish to partition the plane into almost disjoint measurable sets $\Omega_0$, $\Omega_+$, and $\Omega_-$ such that:
    \begin{enumerate}[label=(\roman*)]
        \item $\Omega_0$ has Lebesgue measure $\abs{\Omega_0} = A$,
        
        \item $\Omega_+$ and $\Omega_-$ have Lebesgue measure $\abs{\Omega_+} = \abs{\Omega_-} = \infty$, and
        
        \item for any compact set $K$ and 
        for any other almost disjoint partition $(\tilde{\Omega}_0, \tilde{\Omega}_+, \tilde{\Omega}_-)$ satisfying properties (i) and (ii), which agrees with $(\Omega_0, \Omega_+, \Omega_-)$ outside of $K$, the local perimeter of the interfaces, as measured within $K$ , is at least that of $(\Omega_0, \Omega_+, \Omega_-)$.
    \end{enumerate}
    We will see below that we may rescale so that $A = 2$.

    As posed, this variational problem raises several questions:

    \begin{enumerate}[label=(\roman*)]
        \item What is meant by perimeter? 

        \item Does there exist a minimizer?

        \item If a minimizer exists, is it unique? Are the chambers connected? Is there a line of symmetry? Is $\Omega_0$ necessarily compact?

        \item Are the interfaces between $\Omega_+, \Omega_-$ and $\Omega_0$ smooth $1$-dimensional manifolds?

        \item Is there a singular set where the interfaces fail to be $1$-dimensional manifolds? If so, what does the singular set look like?
    \end{enumerate}

    In answer to question (i), a very general framework in which to work would be the theory of \textit{sets of finite perimeter}. These are sets $E$ for which the characteristic function $\chi_E$ has a distributional derivative representable by an $\R^2$-valued Radon measure with finite total variation; that is, for which the characteristic function is of \textit{(locally) bounded variation}, $\chi_E \in BV_{\mathrm{loc}}(\R^2)$. For the relevant definitions, see Maggi \cite{maggi2012sets}.
    
    Towards a proof within the framework of sets of finite perimeter, assume we have a candidate for such a locally perimeter minimizing cluster. We conjecture that $\Omega_0$ is necessarily compact, and the technique of \textit{Steiner symmetrization} (a measure-preserving, perimeter-diminishing symmetrization with respect to a hyperplane; see Talenti \cite{talenti1993standard}) would allow us to make the following simplifying assumptions:
    \begin{itemize}
         \item \textit{Assumption:} away from $\Omega_0$, the improper chambers $\Omega_+$, $\Omega_-$ meet along a straight line. (This seems intuitively plausible: deviations from straightness must necessarily increase arclength.)
        
        \item \textit{Assumption:} $\Omega_0$ shares boundary with both $\Omega_+$ and $\Omega_-$. (Also intuitively plausible: we can reduce the local arclength by having $\Omega_0$ share boundary with the improper chambers.)
        
        \item \textit{Assumption:} $\Omega_0$ is convex, and is symmetrical about the axis formed by the interface between $\Omega_+$ and $\Omega_-$ away from $\Omega_0$.

        \item \textit{Assumption:} By rotating and translating our coordinate system if necessary, we can assume that $\Omega_0$ is centered at $O = (0,0)$, and that
        \begin{align*}
            \Omega_+ = \{(x,y) : y \geq 0\} \setminus \mathrm{int}(\Omega_0), \\
            \Omega_- = \{(x,y) : y \leq 0\} \setminus \mathrm{int}(\Omega_0),
        \end{align*}
        so that $\Omega_+$ and $\Omega_-$ meet along the $x$-axis away from $\Omega_0$.

        \item \textit{Assumption:} Denote the boundary of $\Omega_0$ by $\partial \Omega_0 = \Gamma_+ \cup \Gamma_-$, where $\Gamma_+ = \partial \Omega_+ \cap \partial \Omega_0$ and $\Gamma_- = \partial \Omega_- \cap \partial \Omega_0$. Then we assume for the purposes of our discussion that the curve $\gamma = \Gamma_+$ is a continuous parametrizable curve symmetric about the $y$-axis, lying in the upper half-space $H = \{y \geq 0\}$; see Figure \ref{fig:labeled-regions}.
    \end{itemize}

    \begin{figure}
        \centering
        \includegraphics[scale=0.33]{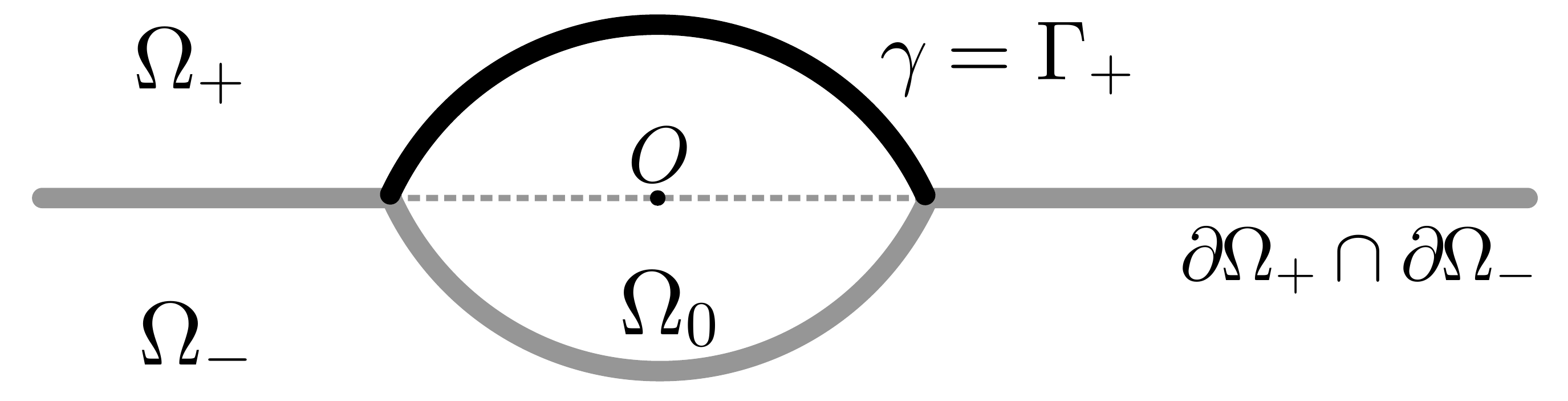}
        \caption{Symmetrized candidate configuration, with $\Omega_0$ centered at the origin $O$, symmetric about the $x$-axis (dashed), and with $\gamma = \Gamma_+$ (in black)}
        \label{fig:labeled-regions}
    \end{figure}
    
    We would therefore like to determine the shape of $\gamma$. \textit{A priori}, we do not know the location of the endpoints of $\gamma = \Gamma_+$, which we may assume are of the form $(\pm p, 0)$ for some $p > 0$. We wish to minimize the length of $\gamma$ and its reflection through the $x$-axis, while simultaneously locally minimizing the portion of the $x$-axis adjacent to $\gamma$, namely $\partial \Omega_+ \cap \partial \Omega_-$ . To achieve the latter, it is equivalent to maximize the distance between the endpoints of $\gamma$. 
    
    Denote the length of $\gamma$ by $\ell(\gamma)$, and the length of $\partial \Omega_0$ by $\ell(\partial \Omega_0)$; let $A(\gamma)$ denote the area enclosed by $\gamma$ and the $x$-axis; and let $d(\gamma) = 2p$ denote the distance along the $x$-axis between the endpoints of $\gamma$. We therefore wish to minimize the functional
    \begin{equation*}
        J[\gamma] = \ell(\gamma) - \frac{1}{2} d(\gamma) = \frac{1}{2}(\ell(\partial \Omega_0) - d(\gamma)),
    \end{equation*}
    for $\gamma = (x,y)$ a sufficiently regular parametrized curve satisfying 
    \begin{equation*}
        y \geq 0, \quad x(t_0) = -x(t_1) = p > 0, \quad A(\gamma) = 1.
    \end{equation*}
    This is a free-endpoint isoperimetric problem in the calculus of variations.

    As it turns out, the solution to this variational problem is well-known: using the theory of sets of finite perimeter, the resolution to this problem appears in Maggi \cite{maggi2012sets}, Thm.\ 19.21, in $n \geq 2$ dimensions and with a relative adhesion coefficient $\beta \in (-1, 1)$. In our special case above, $n = 2$ and  $\beta = \frac{1}{2}$. 

    \subsection{Article Outline}
    
    The goal of this article is to examine what may be said using \textit{classical} methods about existence, uniqueness, and regularity of minimizers in the variational problem addressed by Maggi Thm.\ 19.21, for the special case $n = 2$.
    
    The outline of the article is as follows:
    \begin{itemize}
        \item In Section \ref{sec:C1-graphs}, we present a uniqueness-regularity proof for graphs of $C^1$ functions to obtain a candidate for a minimizer. (Existence in this case can be proved as in Section \ref{sec:W11-graphs}.) 
    
        \item In Section \ref{sec:W11-graphs}, inspired by Talenti \cite{talenti1993standard}, we present an existence-uniqueness-regularity proof for graphs of $W^{1,1}$ functions. This proof is enabled by a useful strict convexity property of the integrand which is gained by restricting to the class of graphs of functions.

        \item In Section \ref{sec:conjectures}, we conclude with an open question and two conjectures related to the class of improper planar partitioning problems.
    \end{itemize}
    \section{Proof for Graphs of \texorpdfstring{$C^1$}{C1} Functions}\label{sec:C1-graphs}

    In this section, we present a classical uniqueness-regularity proof for the variational problem described in the introduction, with competition among graphs of $C^1$ functions. The variational problem is made precise below, with $0 < \beta < 1$.  
     
    \begin{table}
        \centering
        \begin{tabular}{c c}
             $\mathtt{Among}$ & nonnegative $C^1$ functions $u : [-p,p] \to \R$ with $p > 0$ free \\ [1ex]
             $\mathtt{minimize}$ & $\displaystyle J[u] = \int_{-p}^p \left[\sqrt{1+u'^2} - \beta\right] \mathrm{d}x$ \\ [1ex]
             $\mathtt{subject\;to}$ & $\displaystyle K[u] = \int_{-p}^p u \,\mathrm{d}x = 1$ \\ [1ex]
             $\mathtt{and}$ & $u(p) = u(-p) = 0.$
        \end{tabular}
        \label{tab:variational-problem-graphs}
    \end{table}

     We note in particular that this problem is scale invariant under the change of variable $\tilde u=A u$ and $\tilde x=A x$ where $A$ is any prescribed area. This allows us to solve the problem choosing $A=1$. We also remark that scale invariance requires the endpoints to be free; for fixed support $[-p,p]$, existence of a minimizer would require the area to be at most $\frac12 \pi p^2$.
    
    \begin{theorem}
        If $0 < \beta < 1$, then there exists a unique smooth minimizer $u(x)$ for the above variational problem.
        In particular, the unique minimizer $u(x)$ is an arc of a circle of the form
        \begin{equation*}
            u(x) = \sqrt{R^2 - x^2} - \beta R,
        \end{equation*}
        where the radius of curvature $R > 0$ is determined by
        \begin{equation*}
            R^{-1} = \sqrt{\arccos \beta - \beta\sqrt{1-\beta^2}},
        \end{equation*}
        and $u(x)$ is defined on $[-p,p]$ with $p = \sqrt{1 - \beta^2} R$. Furthermore, the graph of $u$ meets the $x$-axis at $\pm p$ with interior angle $\arccos \beta$.
    \end{theorem}
    
    \begin{proof}
        (\textit{Uniqueness}) Assume that a minimizer $u \in C^1[-p,p]$ exists for some $p > 0$ for the above variational problem.
         Let $v \in C^1[-q,q]$ be a competitor curve, nearby to $u$ in the sense of the $C^1$ distance. 
        Since $v$ vanishes at the endpoints of its symmetric domain, the endpoint increments satisfy
        \begin{equation*}
            \delta x(p) = -\delta x(-p), \quad \delta y = 0. 
        \end{equation*}
        
        Since $u$ is a constrained minimizer, there must exist a Lagrange multiplier $\lambda \in \R$ such that the augmented functional
        \begin{equation*}
            \Lambda[u] := J[u] - \lambda K[u]
        \end{equation*}
        is stationary at $u$, i.e., the first variation $\delta \Lambda$ must vanish at $u$. Written out in full, we have
        \begin{equation*}
            \Lambda[u] = \int_{-p}^{p} F(u,u') \,\mathrm{d}x, \quad F(u,u') = \sqrt{1+u'^2} - \beta - \lambda u.
        \end{equation*}
        
        Per Kot \cite{kot2014first}, the first variation $\delta \Lambda$ is given by
        \begin{equation}\label{eq:first-variation}
            \delta \Lambda = \int_{-p}^{p} \left[F_u - \frac{d}{dx} F_{u'}\right] h \,\mathrm{d}x + [F - u' F_{u'}] \delta x \big|_{x = -p}^{x = p}
        \end{equation}
        where $h = v - u$ is defined on $ [-p,p] \cup [-q,q]$, and we extend $u$ and $v$ linearly at the endpoints as needed so that $h$ is well-defined (see e.g. Gelfand-Fomin \cite{gelfand2000calculus}).
    
        Taking variations which fix the support, the integral part of \eqref{eq:first-variation} yields an Euler-Lagrange equation which reads
        \begin{equation*}
            0 = F_u - \frac{d}{dx} F_{u'} = -\lambda - \frac{d}{dx}\left(\frac{u'}{\sqrt{1+{u'}^2}}\right).
        \end{equation*}
        
        Integrating twice, we see that
        \begin{equation*}
            u(x) = \pm \frac{1}{\lambda} \sqrt{1 - (\lambda x + C_1)^2} + C_2
        \end{equation*}
        for some constants of integration $C_1, C_2$. Rearranging and absorbing the appropriate multiplicative constants into $C_1$ and $C_2$, we see that $u(x)$ must satisfy
        \begin{equation*}
            (x + C_1)^2 + (u(x) + C_2)^2 = \left(\frac{1}{\lambda}\right)^2.
        \end{equation*}
        Hence $u(x)$ is an arc of a circle with radius $\frac{1}{\abs{\lambda}}$ and centre $(-C_1, -C_2)$.

        We see that $C_1 = 0$ by symmetry considerations. Indeed, we apply the endpoint conditions $u(p) = u(-p) = 0$ to obtain
        \begin{align*}
            (p + C_1)^2 + C_2^2 &= \frac{1}{\lambda^2} = (-p + C_1)^2 + C_2^2 \\
            \implies \abs{p + C_1} &= \abs{-p + C_1},
        \end{align*}
        from which we conclude that $C_1 = 0$. Hence $u(x)$ is an arc of a circle symmetric about the $y$-axis,
        \begin{equation*}
            x^2 + (u(x) + C_2)^2 = \left(\frac{1}{\lambda}\right)^2.
        \end{equation*}
        Reapplying the endpoint condition $u(\pm p) = 0$, we see that
        \begin{equation*}
            C_2^2 = \left(\frac{1}{\lambda}\right)^2 - p^2 \geq 0 \ \text{ and hence } \ \frac{1}{\abs{\lambda}} \geq p.
        \end{equation*}
        Thus $u(x)$ is given by
        \begin{equation*}
            u(x) = \pm \sqrt{\left( \frac{1}{\lambda}\right)^2 - x^2} \mp \sqrt{\left(\frac{1}{\lambda}\right)^2 - p^2} \quad \text{for all } x \in [-p,p],
        \end{equation*}
        where the signs are such that $u(\pm p) = 0$. Since $u \geq 0$, we must have
        \begin{equation*}
            u(x) = \sqrt{\left(\frac{1}{\lambda}\right)^2 - x^2} - \sqrt{\left(\frac{1}{\lambda}\right)^2 - p^2}.
        \end{equation*}

        Returning to the Euler-Lagrange equation, we can determine the sign of $\lambda$: a computation shows that
        \begin{equation*}
            \frac{u'}{\sqrt{1 + u'^2}} = -\abs{\lambda} x,
        \end{equation*}
        from which we determine that
        \begin{equation*}
            0 = -\lambda - \frac{d}{dx}\left(\frac{u'}{\sqrt{1+u'^2}}\right) = - \lambda - \frac{d}{dx} (-\abs{\lambda}x) = - \lambda + \abs{\lambda},
        \end{equation*}
        so that we have $\abs{\lambda} = \lambda \geq 0$. Hence $u(x)$ is an arc of a circle with radius $\frac{1}{\lambda}$ and centre $(0, \sqrt{(1/\lambda)^2 - p^2})$, where $\lambda \geq 0$ and $\frac{1}{\lambda} \geq p > 0$. The remaining two parameters are determined by the area constraint and the transversality condition.

        We now apply the area constraint. Computing the area functional using the change of variable $x = \frac{1}{\lambda} \sin\theta$, we have:
        \begin{align*}
            K[u] &= \int_{-p}^p \sqrt{\left(\frac{1}{\lambda}\right)^2 - x^2}\,\mathrm{d}x - \int_{-p}^p \sqrt{\left(\frac{1}{\lambda}\right)^2 - p^2} \,\mathrm{d}x \\
            &= \frac{2}{\lambda^2} \int_0^{\arcsin(\lambda p)} \cos^2 \theta \,\mathrm{d}\theta - 2p\sqrt{\left(\frac{1}{\lambda}\right)^2 - p^2} \\
            &= \frac{1}{\lambda^2} \arcsin(\lambda p) - p\sqrt{\left(\frac{1}{\lambda}\right)^2 - p^2}.
        \end{align*}
        We note that this is the area of the circular segment with arc $u(x)$ and chord the $x$-axis. Hence the area constraint reads
        \begin{equation}\label{eq:area-constraint-graphs}
            1 = \frac{1}{\lambda^2} \arcsin(\lambda p) - p \sqrt{\left(\frac{1}{\lambda}\right)^2 - p^2}.
        \end{equation}
        This is a transcendental equation in the variables $\lambda, p$, and is well-defined since $\frac{1}{\lambda} \geq p$. To progress, we need the transversality condition.

        Next from \eqref{eq:first-variation}, by varying the end points, we obtain the transversality condition
        \begin{equation*}
            [F(u,u') - u' F_{u'}(u,u')] \delta x \Big|_{x = -p}^{x = p} = 0\;.
        \end{equation*}
        From the definition of the integrand $F$, we find that
        \begin{equation*}
            F - u' F_{u'} = \frac{1}{\sqrt{1 + u'^2}} - \beta - \lambda u.
        \end{equation*}
        Furthermore, we have that $[u'(p)]^2 = [u'(-p)]^2$. Since $u(\pm p) = 0$, the transversality condition yields the \textit{angle condition}
        \begin{equation}\label{eq:angle-condition-graphs}
            \frac{1}{\sqrt{1+u'(p)^2}} = \beta.
        \end{equation}
        Note that this shows the tangent vector $(1, u'(p))$ makes an angle of $\arccos \beta$ with the positive $x$-axis, as claimed. (Note also that this implies $0 < \beta < 1$, so our assumption was necessary.) From the definition of $u'(x)$, we find that 
        \begin{equation*}
            \frac{1}{\sqrt{1 + u'(p)^2}} = \sqrt{1 - (\lambda p)^2} = \beta,
        \end{equation*}
        from which we obtain
        \begin{equation}\label{eq:endpoint-value-graphs}
            \lambda p = \sqrt{1 - \beta^2}.
        \end{equation}
        Inserting this into the area constraint \eqref{eq:area-constraint-graphs}, we find that
        \begin{equation*}
            \lambda^2 = \arcsin(\lambda p) - \lambda p \sqrt{1 - (\lambda p)^2} = \arccos\beta - \beta \sqrt{1 - \beta^2},
        \end{equation*}
        where we have used the identity $\arcsin \sqrt{1-\theta^2} = \arccos \theta$.
        
        Define $R = \frac{1}{\lambda}$, the radius of the circular arc. We may then identify $\lambda$ as the \textit{curvature} of the arc, so that $R = \frac{1}{\lambda}$ is the \textit{radius of curvature}. Written out explicitly, we have
        \begin{equation*}\label{eq:minimizer-curvature-graphs}
            R^{-1} = \sqrt{\arccos \beta - \beta\sqrt{1 - \beta^2}}.
        \end{equation*}
        From \eqref{eq:endpoint-value-graphs}, it follows that $p = \sqrt{1 - \beta^2} R$ and $\sqrt{R^2 - p^2} = \beta R$.
        
        Therefore the minimizer $u(x)$ must have the form
        \begin{equation*}
            u(x) = \sqrt{R^2 - x^2} - \beta R,
        \end{equation*}
        where $R$ is defined by \eqref{eq:minimizer-curvature-graphs}, and $u(x)$ is defined on $[-p,p]$ where $p = \sqrt{1 - \beta^2} R$. This proves uniqueness.

        (\textit{Existence}) This can be proved as in Section \ref{sec:W11-graphs}, so we omit the proof here.

        (\textit{Regularity}) We have shown that 
        \begin{equation*}
            u(x) = \sqrt{R^2 - x^2} - \beta R
        \end{equation*}
        is the unique candidate for a minimizer of our variational problem. The smoothness of $u(x)$ follows immediately, since $\abs{x} \leq p < R$ and $u(x)$ is a composition of smooth functions. \qed
    \end{proof}

    \begin{remark}\label{rmk:angle-condition}
        With respect to the parametrization $t \mapsto (t, u(t))$, the graph of $u(x)$ has (non-unit) tangent vector $(1, u'(x))$. Normalizing and rotating the tangent vector $90^\circ$ counter-clockwise, we find that the outward unit normal $\nu$ of the graph of $u(x)$ is given by
        \begin{equation*}
            \nu = \left(- \frac{u'(x)}{\sqrt{1 + u'(x)^2}}, \frac{1}{\sqrt{1 + u'(x)^2}}\right).
        \end{equation*}
        As such, we see that \eqref{eq:angle-condition-graphs} tells us that
        \begin{equation*}
            \nu \cdot e_2 = \beta \quad \text{at } x = \pm p,
        \end{equation*}
        in agreement with the modern statement of the theorem in Maggi \cite{maggi2012sets}. For our special case of $\beta = \frac{1}{2}$, we obtain an interior angle of $60^\circ$, so the exterior angle is $120^\circ$ as in Figure \ref{fig:optimal-solution}.
        
         Further we note that \eqref{eq:angle-condition-graphs} is consistent with our assumption that $0<\beta<1$. The case $\beta=1$ is impossible, as then $u'(\pm p)=0$, which yields only a trivial solution by ODE uniqueness and hence cannot satisfy the area condition. The case $\beta=0$ is classical and the minimizer is a half disk meeting the axis at vertical tangents. We could consider other values of $\beta$ by enlarging our minimization problem as in Maggi \cite{maggi2012sets} over the class of finite perimeter sets in the upper half plane. Indeed, in case $\beta\in (-1,0)$, Maggi's theorem applies and we obtain curves that meet the $x$-axis at an angle given by $\arccos \beta>\pi/2$ (see \cite[Theorem 19.21]{maggi2012sets}).  These are not graphs, and in fact the functional $J$ admits no minimizer in this regime. In case $\beta\le -1$, it is easy to verify that the minimizer is a ball disjoint from the axis (see Maggi \cite[Remark 19.20]{maggi2012sets}).  When $\beta\ge 1$ it is shown in Maggi \cite[Remark 19.19]{maggi2012sets} that the functional is unbounded below, and there is no minimizer even in the general setting. We note that it is also proved in \cite{maggi2012sets} Section 19.1.1 that $|\beta|<1$ is a necessary condition for lower semicontinuity.
        
    \end{remark}
    \section{Proof for Graphs of \texorpdfstring{$W^{1,1}$}{W11} Functions}\label{sec:W11-graphs}

    We wish to solve the variational problem below, with $0 < \beta < 1$. Following the approach presented in Talenti \cite{talenti1993standard}, we prove an isoperimetric inequality, from which we establish the existence of a unique smooth minimizer of the following variational problem:
    \begin{table}
        \centering
        \begin{tabular}{c c}
             $\mathtt{Among}$ & nonnegative functions $u \in W^{1,1}[-p,p]$ with $p > 0$ free\\ [1ex]
             $\mathtt{minimize}$ & $\displaystyle J[u] = \int_{-p}^p \left[\sqrt{1+u'^2} - \beta\right] \,\mathrm{d}x$ \\ [1ex]
             $\mathtt{subject\;to}$ & $\displaystyle K[u] = \int_{-p}^p u \,\mathrm{d}x = 1$ \\ [1ex]
             $\mathtt{and}$ & $u(p) = u(-p) = 0.$
        \end{tabular}
               
        \label{tab:variational-problem-sobolev}
    \end{table}  

    Taken on its own, the proof below is perhaps a little philosophically unsatisfying, as it is essentially a mathematical sleight of hand, which avoids the direct method of convergence of minimizing sequences. However, the hard work of identifying a candidate for a minimizer was carried out in Section \ref{sec:C1-graphs}, so we are free to pull the rabbit out of a hat!
    
    \begin{theorem}[Isoperimetric Inequality]\label{isoperimetric-inequality}
        Let $0 < \beta < 1$, and let $u$ be a nonnegative real-valued function in $W^{1,1}[-p,p]$ defined on an interval $[-p,p]$, with $p > 0$ a free parameter. Assume $u$ vanishes at both endpoints, i.e.,
        \begin{equation*}
            u(-p) = u(p) = 0.
        \end{equation*}
        Define the length of the graph of $u$ and the area under the graph of $u$ by
        \begin{equation*}
            L = \int_{-p}^p \sqrt{1 + u'^2} \,\mathrm{d}x, \quad  A = \int_{-p}^p u \,\mathrm{d}x
        \end{equation*}
        respectively. Then
        \begin{equation}\label{eq:our-isoperimetric-inequality}
            L \geq \left(\frac{\arccos\beta}{\sqrt{1 - \beta^2}} + \beta\right) p + \sqrt{1 - \beta^2} \frac{A}{p}
        \end{equation}
        with equality if and only if the graph of $u$ is an arc of a circle of the form
        \begin{equation*}
            u(x) = \sqrt{R^2 - x^2} - \beta R,
        \end{equation*}
        where $R$ is determined by
        \begin{equation*}
            R = \frac{p}{\sqrt{1 - \beta^2}}.
        \end{equation*}
    \end{theorem}   
    
    \begin{proof}
        Let $p > 0$ be fixed, and consider the functional
        \begin{equation*}
            \Lambda[u,p] = \frac{\sqrt{1 - \beta^2}}{p} \int_{-p}^{p} \left[\sqrt{1 + u'^2} - \beta + \frac{\sqrt{1 - \beta^2}}{p} x u'\right] \,\mathrm{d}x.
        \end{equation*}
        Note that the integrand $f(x,u') = \sqrt{1 + (u')^2} - \beta + \tfrac{\sqrt{1-\beta^2}}{p} x u'$ is strictly convex with respect to $u'$: we have
        \begin{equation*}
            \frac{\partial^2}{\partial(u')^2} f(x,u') = \frac{1}{(1 + u'^2)^{\tfrac{3}{2}}} > 0.
        \end{equation*}
        Thus for any nonnegative $u, w \in W^{1,1}[-p,p]$ with $u(p) = u(-p) = 0$ and $w(p) = w(-p) = 0$, we have
        \begin{equation*}
            f(x,w') \geq f(x,u') + (w' - u') \frac{\partial}{\partial(u')}f(x,u')
        \end{equation*}
        with equality if and only if $w' = u'$.

        Define
        \begin{equation}\label{eq:unique-minimizer}
            R = \frac{p}{\sqrt{1 - \beta^2}} \quad\text{and}\quad u(x) = \sqrt{R^2 - x^2} - \beta R.
        \end{equation}
        Note since $0 < p < R$ that $u(x)$ is smooth up to the boundary of $[-p,p]$, so $u \in W^{1,1}[-p,p]$. Furthermore, we have $u(p) = u(-p) = 0$ by construction, and
        \begin{align*}
            \frac{\partial}{\partial u'} f(x, u') = \frac{u'}{\sqrt{1 + u'^2}} + \frac{\sqrt{1 - \beta^2}}{p} x = \frac{-x}{R} + \frac{x}{R} = 0.
        \end{align*}
        Therefore for any $w \in W^{1,1}[-p, p]$ with $w(p) = w(-p) = 0$, we have
        \begin{equation*}
            f(x,w') \geq f(x,u')
        \end{equation*}
        with equality if and only if $w' = u'$. Since $w$ and $u$ both vanish at the endpoints, $w' = u'$ holds if and only if $w = u$. As such, we have
        \begin{equation*}
            \Lambda[w,p] = \frac{\sqrt{1 - \beta^2}}{p} \int_{-p}^p f(x,w') \,\mathrm{d}x \geq \frac{\sqrt{1 - \beta^2}}{p} \int_{-p}^p f(x,u') \,\mathrm{d}x = \Lambda[u,p]
        \end{equation*}
        with equality if and only if $w = u$. Therefore $u$ is the unique minimizer of the functional $\Lambda[u,p]$. \textit{A posteriori}, we see that $u$ is smooth on $[-p, p]$, as noted above.
        
        We wish to compute the value of the functional attained by the unique minimizer $u(x)$ given by \eqref{eq:unique-minimizer}. Recalling that $p = \sqrt{1 - \beta^2}R$, we have
        \begin{align*}
            \Lambda[u] 
            &= \frac{\sqrt{1-\beta^2}}{p} \int_{-p}^p f(x,u') \,\mathrm{d}x \\
            &= \frac{1}{R} \int_{-p}^{p} \left[\sqrt{1 + u'^2} - \beta + \frac{1}{R} x u'\right] \,\mathrm{d}x \\
            &= \frac{1}{R} \int_{-p}^{p} \left[\frac{R}{\sqrt{R^2 - x^2}} - \beta + \frac{1}{R} x \left(-\frac{x}{\sqrt{R^2 - x^2}}\right)\right] \,\mathrm{d}x \\
            &= \frac{1}{R} \int_{-p}^{p} \left[\frac{R^2 - x^2}{R\sqrt{R^2 - x^2}} - \beta\right] \,\mathrm{d}x = \frac{1}{R} \int_{-p}^{p} \left[\frac{\sqrt{R^2 - x^2}}{R} - \beta\right] \,\mathrm{d}x \\
            &= \frac{2}{R^2} \int_0^p \sqrt{R^2 - x^2} \,\mathrm{d}x - 2\beta\sqrt{1-\beta^2}.
        \end{align*}
        The remaining integral can be computed using the change of variables $x = R \sin\theta$. Note that $\arcsin \sqrt{1 - \beta^2} = \arccos \beta$. We have
        \begin{align*}
            \int_0^p \sqrt{R^2 - x^2} \,\mathrm{d}x
            &= \int_0^{\arcsin(p / R)} \sqrt{R^2 - R^2 \sin^2\theta} R \cos \theta \,\mathrm{d}\theta \\
            &= R^2 \int_0^{\arcsin\sqrt{1-\beta^2}} \cos^2 \theta \,\mathrm{d}\theta = \frac{R^2}{2} \int_0^{\arccos\beta} [1 + \cos(2\theta)] \,\mathrm{d}\theta \\
            &= \frac{R^2}{2} \left[\theta + \frac{1}{2} \sin(2\theta)\right]\big|_{0}^{\arccos\beta} = \frac{R^2}{2} \left[\theta + \sin(\theta)\cos(\theta)\right]\big|_{0}^{\arccos\beta} \\
            &= \frac{R^2}{2} \left[\arccos\beta + \sqrt{1 - \beta^2} \beta \right].
        \end{align*}
        Inserting this into our expression for $\Lambda[u,p]$, we obtain
        \begin{equation*}
            \Lambda[u,p]
            = \arccos\beta - \beta\sqrt{1 - \beta^2}.
        \end{equation*}
        
        Thus for all nonnegative $w \in W^{1,1}[-p,p]$ with $w(-p) = w(p) = 0$, we have
        \begin{equation*} \label{eq:sharp-lower-bound-on-constrained-functional}
            \Lambda[w,p] \geq \arccos\beta - \beta\sqrt{1-\beta^2},
        \end{equation*}
        with equality if and only if $w = u$ as given in \eqref{eq:unique-minimizer}.
        
        Integrating the functional $\Lambda[w,p]$ by parts, the inequality \eqref{eq:sharp-lower-bound-on-constrained-functional} yields
        \begin{equation*}
            \frac{\sqrt{1-\beta^2}}{p} \int_{-p}^p \left[\sqrt{1 + w'^2} - \beta - \frac{\sqrt{1-\beta^2}}{p} w\right] \,\mathrm{d}x \geq \arccos\beta - \beta\sqrt{1-\beta^2},
        \end{equation*}
        with equality if and only if $w = u$. Written in terms of $L$ and $A$, we have
        \begin{equation*}
            \frac{\sqrt{1 - \beta^2}}{p} \left[L - 2\beta p - \frac{\sqrt{1-\beta^2}}{p} A\right]\geq \arccos\beta - \beta\sqrt{1-\beta^2},
        \end{equation*}
        with equality if and only if $w = u$. Rearranging, we obtain the isoperimetric inequality
        \begin{equation*}
            L \geq \left(\frac{\arccos\beta}{\sqrt{1 - \beta^2}} + \beta\right) p + \sqrt{1 - \beta^2} \frac{A}{p},
        \end{equation*}
        with equality if and only if $w = u$ as given in \eqref{eq:unique-minimizer}, as was to be shown. \qed
    \end{proof}
    
    \begin{theorem}[Existence-Uniqueness-Regularity in $W^{1,1}$]
        For our variational problem, there exists a unique smooth minimizer, given by
        \begin{equation*}
            u(x) = \sqrt{R^2 - x^2} - \beta R,
        \end{equation*}
        where $R$ is determined by
        \begin{equation*}
            R^{-1} = \sqrt{\arccos\beta - \beta\sqrt{1 - \beta^2}}
        \end{equation*}
        and $p = \sqrt{1 - \beta^2} R$. Furthermore, the graph of $u$ meets the $x$-axis at $\pm p$ with interior angle $\arccos \beta$.
    \end{theorem}
    
    \begin{proof}
        Observe that the functional $J[u]$ in our problem has the form
        \begin{equation*}
            J[u] = \int_{-p}^p [\sqrt{1 + u'^2} - \beta] \,\mathrm{d}x = L - 2\beta p.
        \end{equation*}
        By the isoperimetric inequality \eqref{eq:our-isoperimetric-inequality}, we have that
        \begin{equation}\label{eq:lower-bound-on-objective-functional}
            L - 2\beta p \geq \left(\frac{\arccos\beta}{\sqrt{1-\beta^2}} - \beta\right) p + \sqrt{1 - \beta^2} \frac{A}{p},
        \end{equation}
        with equality if and only if $u$ is given by \eqref{eq:unique-minimizer}.
        
        We wish to determine when the right-hand side of \eqref{eq:lower-bound-on-objective-functional} is minimized. For the function $g(p)$ which defines the right-hand side, we have
        \begin{align*}
            g'(p) = \left(\frac{\arccos\beta}{\sqrt{1-\beta^2}}-\beta\right) - \sqrt{1-\beta^2}\frac{A}{p^2}, \\
            g''(p) = 2\sqrt{1-\beta^2} \frac{A}{p^3} > 0 \quad\text{for all } p > 0,
        \end{align*}
        so $g(p)$ is concave up. Furthermore, $g'(p)$ vanishes precisely when
        \begin{equation*}
            p = p_0 := \sqrt{\frac{(1-\beta^2) A}{\arccos\beta - \beta\sqrt{1-\beta^2}}}.
        \end{equation*}
        Since $g(p)$ is concave up, it follows that $p = p_0$ yields a minimum of $g(p)$. The minimum value is then
        \begin{equation*}
            g(p_0) = 2 \sqrt{A(\arccos\beta - \beta\sqrt{1-\beta^2})}.
        \end{equation*}

        Therefore for $p > 0$ free and $w \in W^{1,1}[-p,p]$ with $w(p) = w(-p) = 0$, we have
        \begin{equation*}
            L - 2\beta p \geq 2\sqrt{A(\arccos \beta - \beta \sqrt{1 - \beta^2})},
        \end{equation*}
        with equality if and only if $p = p_0$ and $w = u$ with
        \begin{equation*}
            u(x) = \sqrt{R^2 - x^2} - \beta R, \quad R = \frac{p_0}{\sqrt{1 - \beta^2}} = \sqrt{\frac{A}{\arccos\beta - \beta\sqrt{1-\beta^2}}}.
        \end{equation*}
        
        As such, we see that among nonnegative functions $w \in W^{1,1}[-p,p]$ with $p > 0$ free, $w(p) = w(-p) = 0$, and $A = 1$, the functional $J[u] = L - 2\beta p$ has the unique minimizer
        \begin{equation*}
            u(x) = \sqrt{R^2 - x^2} - \beta R, \quad \text{with } R^{-1} = \sqrt{\arccos\beta - \beta\sqrt{1 - \beta^2}},
        \end{equation*}
        which \textit{a posteriori} is smooth up to the boundary on the interval $[-p,p]$ with
        \begin{equation*}
            p = \sqrt{1-\beta^2} R.
        \end{equation*}
        Thus our variational problem has a unique smooth minimizer, as was to be shown. The angle condition is satisfied per Remark \ref{rmk:angle-condition}. 
        \qed
    \end{proof}

    \begin{remark} If we assume more generally we have competition among non-negative $w\in W^{1,1}[a,b]$ for any support $[a,b]$, we may assume by extending $w$ by zero that $w\in W^{1,1}[-p,p]$ with $p=\max\{|a|,|b|\}$. Note that the extension increases the value of $J[w]$ and so the above argument can be extended to minimization of $J[u]$ among all compactly supported $W^{1,1}$ functions.
    \end{remark}

    \section{Further Conjectures}\label{sec:conjectures}

    In this article, we set out to answer the following question: can we restrict the setting of Maggi Thm.\ 19.21 to one where geometric measure theory is not necessary and still say something meaningful about existence, uniqueness, and regularity of minimizers?
    
    It was not clear \textit{a priori} whether the answer was yes, but in Sections \ref{sec:C1-graphs} and \ref{sec:W11-graphs} we saw that restricting the setting to the class of curves which are expressible as graphs of $C^1$ or $W^{1,1}$ functions gives the augmented functional a useful strict convexity property. In these restricted settings, a full proof of existence-uniqueness-regularity is possible using only classical (i.e.\ non-direct) methods, adapting work of Talenti \cite{talenti1993standard}. This begs the question: can Maggi Thm.\ 19.21 be proved using ``classical" techniques in dimension $n > 2$?

    In connection with the class of improper partitioning problems described in Section \ref{sec:intro}, we conclude with an open question and two conjectures related to this research:
    \begin{enumerate}[label=(\roman*)]
        \item For the improper planar partitioning problem with $N = 1$ chamber and $M = 2$ improper chambers, we conjectured that $\Omega_0$ is necessarily compact and that Steiner symmetrization would allow us to reduce to the variational problem solved by Maggi Thm.\ 19.21. Is this correct?

        \item For the improper planar partitioning problem with $N = 1$ chamber and $M = 3$ improper chambers, we conjecture  that the unique locally perimeter minimizing configuration (up to planar isometries) is the one depicted in Figure \ref{fig:one-compact-three-unbounded}. This configuration is given by stereographic projection of the standard $4$-bubble on the $2$-sphere, with the point at infinity placed at one of the junctions.  It also agrees with the Steiner partition of $\R^2$ outside of sufficiently large compact sets.

        \begin{figure}
            \centering
            \includegraphics[scale=0.5]{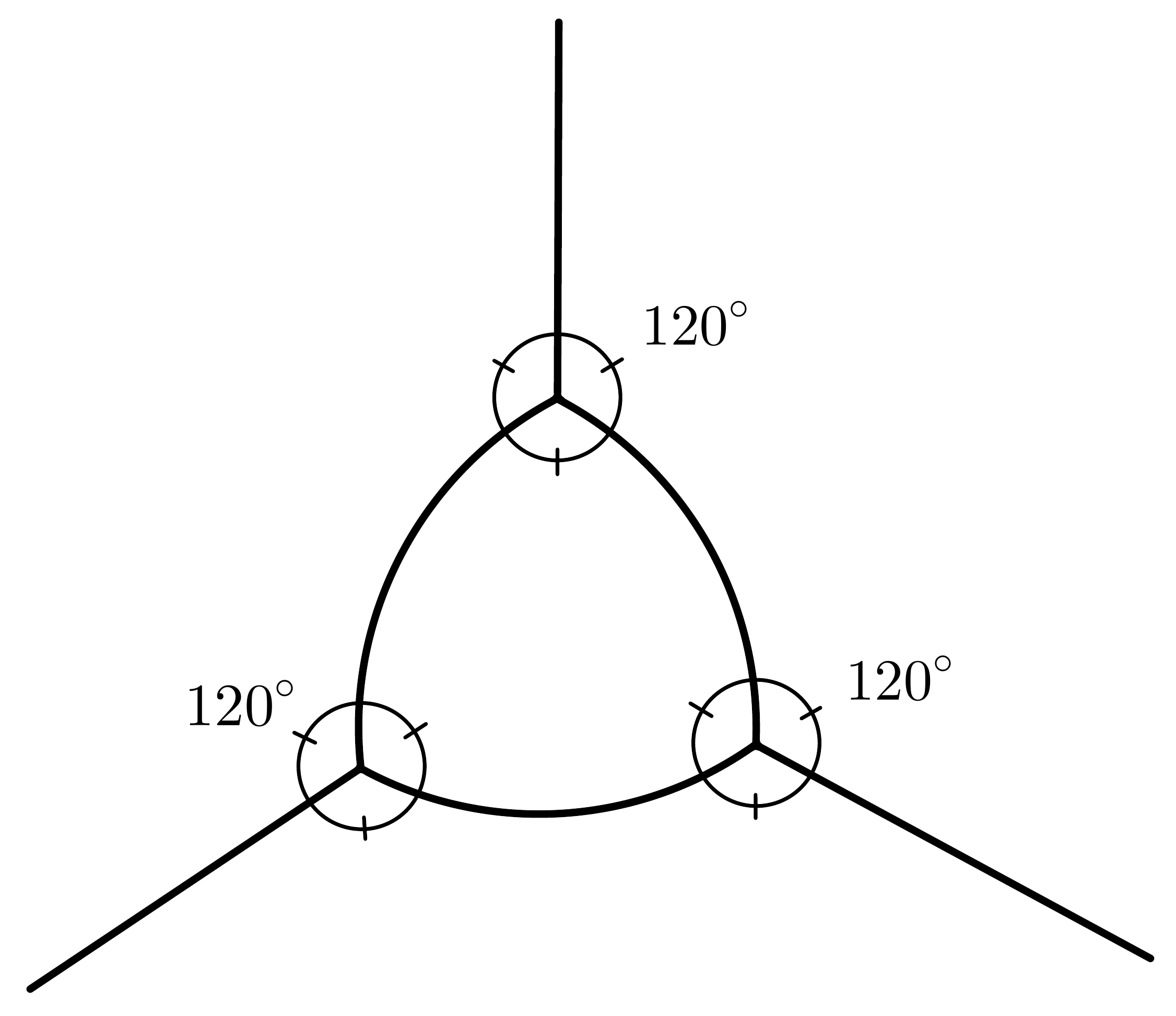}
            \caption{Conjectured optimal configuration for the improper planar partitioning problem with one chamber and three improper chambers}
            \label{fig:one-compact-three-unbounded}
        \end{figure}

        \item For the improper planar partitioning problem with $N = 2$  chambers of equal area and $M = 2$ improper chambers, we conjecture  that the unique locally perimeter minimizing configuration (up to planar isometries) is the one depicted in Figure \ref{fig:two-compact-two-unbounded}. This configuration is given by stereographic projection of the standard $4$-bubble on the $2$-sphere, with the point at infinity placed at the midpoint of one of the interfaces.  It also agrees with the minimizing cone-like 2-cluster in $\R^2$ outside of sufficiently large compact sets.

        \begin{figure}
            \centering
            \includegraphics[scale=0.4]{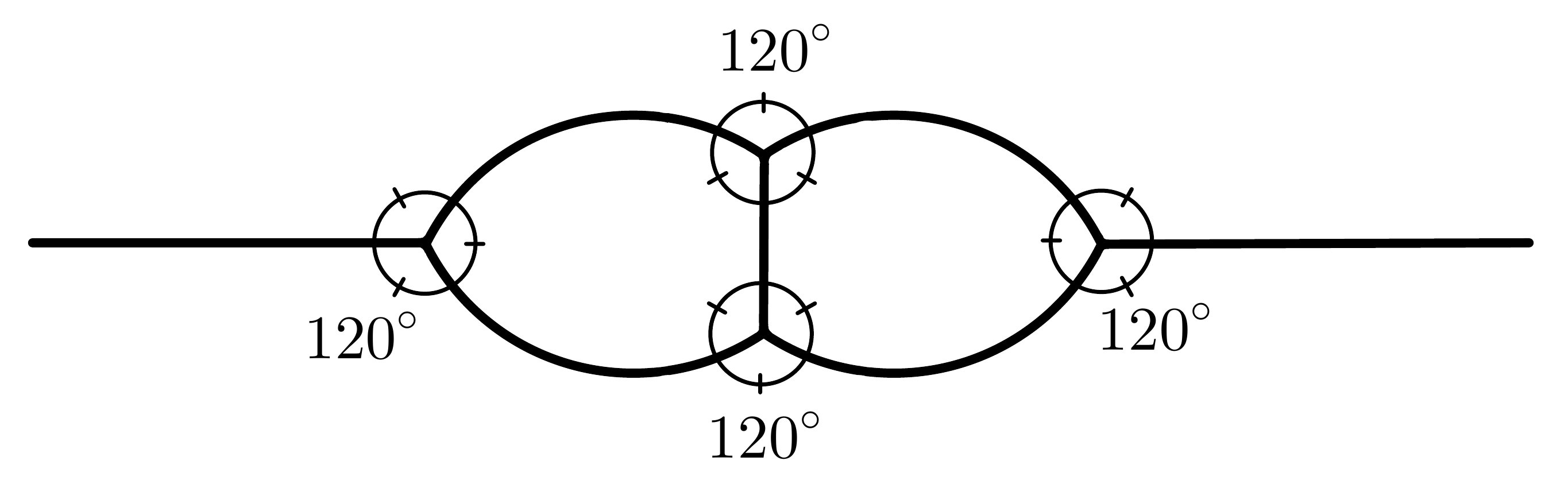}
            \caption{Conjectured optimal configuration for the improper planar partitioning problem with two chambers of equal area and two improper chambers}
            \label{fig:two-compact-two-unbounded}
        \end{figure}




    \end{enumerate}

\end{document}